\documentclass[11pt,reqno]{amsart}
\usepackage[T1]{fontenc}
\usepackage{lmodern}
\usepackage[letterpaper,textwidth=6.35in,textheight=8.8in,centering]{geometry}
\usepackage{mathtools,amssymb,booktabs,array}
\usepackage[expansion=false]{microtype}
\usepackage{enumitem}
\usepackage{listings}
\usepackage[hidelinks]{hyperref}
\hypersetup{pdftitle={A local cohomology obstruction to small Cohen--Macaulay modules},pdfauthor={Liang Chen},pdfsubject={Proposed proof for independent mathematical verification}}
\numberwithin{equation}{section}
\newtheorem{theorem}{Theorem}[section]
\newtheorem{proposition}[theorem]{Proposition}
\newtheorem{lemma}[theorem]{Lemma}
\newtheorem{corollary}[theorem]{Corollary}
\newtheorem{mainclaim}[theorem]{Main claim}
\newtheorem{consequence}[theorem]{Consequence}
\theoremstyle{definition}

\theoremstyle{remark}
\newtheorem{remark}[theorem]{Remark}
\newcommand{\C}{\mathbb C}
\newcommand{\Z}{\mathbb Z}
\newcommand{\PP}{\mathbb P}
\newcommand{\A}{\mathbb A}
\newcommand{\OO}{\mathcal O}
\newcommand{\FF}{\mathcal F}
\newcommand{\GG}{\mathcal G}
\newcommand{\QQ}{\mathcal Q}
\newcommand{\mm}{\mathfrak m}
\newcommand{\bb}{\mathfrak b}
\newcommand{\rk}{\operatorname{rank}}
\newcommand{\Spec}{\operatorname{Spec}}

\newcommand{\Tot}{\operatorname{Tot}}
\newcommand{\Bl}{\operatorname{Bl}}
\newcommand{\depth}{\operatorname{depth}}

\newcommand{\Ass}{\operatorname{Ass}}
\newcommand{\Tor}{\mathcal{T}or}
\newcommand{\smax}{s_{\max}}
\newcommand{\smin}{s_{\min}}
\newcommand{\Hom}{\operatorname{Hom}}
\newcommand{\Ext}{\operatorname{Ext}}
\newcommand{\im}{\operatorname{im}}

\newcommand{\len}{\ell}

\title[An obstruction to small Cohen--Macaulay modules]{A local cohomology obstruction to small\newline Cohen--Macaulay modules}
\author[Liang Chen]{Liang Chen}
\address{School of Mathematical Sciences, Capital Normal University,
Beijing 100048, China}
\email{2210501003@cnu.edu.cn}

\begin{document}
\begin{abstract}
We propose a local-cohomology obstruction to the existence of small
Cohen--Macaulay modules over completed section rings of surfaces.
Let $Y$ be a smooth connected complex projective surface, and let $H$ be an
ample globally generated divisor with $K_Y\equiv4H$ numerically. Put
$c(Y,H)=15H^2/8-\chi(\OO_Y)$. The main claim is that, when
$c(Y,H)>0$, every nonzero finite reflexive module $M$ over the completed
vertex local ring of $\bigoplus_{n\geq0}H^0(Y,\OO_Y(nH))$ satisfies
$\dim_\C H^2_\mm(M)\geq c(Y,H)\rk M$.
The argument combines elementary transformations of reflexive extensions,
Harder--Narasimhan slopes, Bogomolov's inequality, and a two-section
Koszul estimate. A fixed-source kernel estimate transfers the resulting
cohomology to the punctured spectrum without requiring a grading on $M$.
For the degree-six member of the six-line Hirzebruch--Kummer family,
$H^2=1080$ and $\chi(\OO_Y)=1926$, so the bound is $99\rk M$.
We give an explicit complete-intersection model for this application and
spell out the argument leading to the claimed nonexistence of nonzero
finite maximal Cohen--Macaulay modules.
\end{abstract}
\maketitle

\section{Introduction and the main statement}\label{sec:introduction}

A \emph{small Cohen--Macaulay module} over a Noetherian local ring
$(R,\mm)$ is a nonzero finitely generated $R$-module $M$ such that
$\depth_R M=\dim R$. The small Cohen--Macaulay module conjecture asks
whether every complete Noetherian local domain has such a module.
Finite generation is essential: the existence of big Cohen--Macaulay
modules is a different question. We use the terminology of
Bhatt--Hochster--Ma \cite[Section 1.3]{BHM2026}.

The distinction between local and graded existence is equally important.
For a polarized smooth projective variety, graded maximal Cohen--Macaulay
modules over its section ring are related to bundles with vanishing
intermediate cohomology. A finite module over the completed vertex local
ring need not arise from a graded module. Consequently, a graded
nonexistence theorem does not, by itself, settle the local conjecture.
Anghel's preprint \cite{Anghel2026} gives graded nonexistence results and
explicitly retains this distinction. Its six-line Hirzebruch--Kummer
family supplies the polarized surface used below; neither the surface
nor that graded nonexistence theorem is claimed as new here.

This manuscript records a proposed argument addressing arbitrary finite
modules over the completed local ring. Its central point is to work with
reflexive extensions on a resolution and to change the extension while
keeping its restriction to the punctured spectrum fixed. The required
changes are elementary transformations along the exceptional surface.
After a slope normalization, a negative Euler characteristic forces some
first cohomology to survive all allowed poles along the exceptional
divisor. We give the lemmas and their proofs separately so that this
passage can be checked independently of the numerical example.

\begin{mainclaim}\label{thm:main}
Let $Y$ be a smooth connected complex projective surface, and let $H$ be
an ample globally generated Cartier divisor on $Y$. Assume that
\begin{equation}\label{eq:main-hypotheses}
 K_Y\equiv4H,
 \qquad
 c(Y,H):=\frac{15}{8}H^2-\chi(\OO_Y)>0.
\end{equation}
Set
\begin{equation}\label{eq:section-ring}
 B=\bigoplus_{n\geq0}H^0(Y,\OO_Y(nH)),\qquad
 \bb=B_{>0},\qquad R=\widehat{B_\bb},\qquad \mm=\bb R.
\end{equation}
Then $R$ is a three-dimensional complete normal local domain. For every
nonzero finitely generated reflexive $R$-module $M$, the module
$H^2_\mm(M)$ has finite length, and one has
\begin{equation}\label{eq:main-bound}
 \dim_\C H^2_\mm(M)\ \geq\ c(Y,H)\rk_R M.
\end{equation}
In particular, $R$ admits no nonzero finitely generated maximal
Cohen--Macaulay module.
\end{mainclaim}

The proposed proof of Main claim~\ref{thm:main} appears in
Section~\ref{sec:main-proof}. In the present setting the module
$H^2_\mm(M)$ has finite length, as shown in
Lemma~\ref{lem:finite-deficiency}. Thus the left-hand side of
\eqref{eq:main-bound} is its ordinary finite dimension over $\C$,
equivalently its length over $R$.

Here is the explicit application. In $\PP^5_\C$, let
\begin{equation}\label{eq:W}
 W=V\bigl(u_3^6-u_0^6+u_1^6,\;
          u_4^6-u_0^6+u_2^6,\;
          u_5^6-u_1^6+u_2^6\bigr).
\end{equation}
Let $\rho:Y\to W$ be the blow-up of its singular points, let
$A=\rho^*\OO_W(1)$, and let $E$ be the reduced exceptional divisor.
The geometric assertions in the next statement are verified in
Section~\ref{sec:example}.

\begin{consequence}[Assuming Main claim~\ref{thm:main}]\label{cor:example}
For the surface in \eqref{eq:W}, the divisor $H=3A-E$ is ample and
globally generated, and
\begin{equation}\label{eq:example-numbers}
 K_Y\sim4H,\qquad H^2=1080,\qquad \chi(\OO_Y)=1926.
\end{equation}
Let $R$ be its completed section ring as in \eqref{eq:section-ring}.
Then every nonzero finite reflexive $R$-module satisfies
\begin{equation}\label{eq:99}
 \dim_\C H^2_\mm(M)\geq99\rk_R M.
\end{equation}
Consequently, $R$ has depth $2$, is quasi-Gorenstein, and has an isolated
singularity, but has no nonzero finite maximal Cohen--Macaulay module.
Subject to Main claim~\ref{thm:main}, this ring therefore gives a
counterexample to the small Cohen--Macaulay module conjecture.
\end{consequence}

\begin{remark}[Numerical and graded obstructions]\label{rem:numerical}
The hypothesis on the canonical divisor in Main claim~\ref{thm:main}
is numerical. Its uses are the equalities
$\nu\cdot K_Y=4\nu\cdot H$ in the surface Riemann--Roch estimate
and $s(\omega_Y)=K_Y\cdot H/H^2=4$ in the dual-slope estimate.
No isomorphism $\omega_Y\simeq L^4$ is needed for that assertion.
The stronger linear equivalence in \eqref{eq:example-numbers}
will be used separately to identify the canonical module in
Consequence~\ref{cor:example}.

Set $d=H^2$ and $\sigma(Y)=K_Y^2-8\chi(\OO_Y)$, as in
\cite[Section 1 and Remark 1.6]{Anghel2026}. Under $K_Y\equiv4H$,
\[
 8c(Y,H)=15d-8\chi(\OO_Y)=\sigma(Y)-d.
\]
Consequently $c(Y,H)>0$ is exactly the degree--signature inequality
$d<\sigma(Y)$ from Anghel's graded nonexistence criterion, restricted
to this numerical subcanonical class. The additional assertion proposed
here concerns arbitrary finite modules after localization and completion;
it is not a consequence of the graded criterion alone.
Independently of Main claim~\ref{thm:main},
Lemma~\ref{lem:noncyclic-deficiency} gives
$\len_RH^2_\mm(R)\geq234$ for the explicit ring and shows that its second
deficiency module is not cyclic. Remark~\ref{rem:existence-criteria}
explains why the specific existence criteria of Shimomoto--Tavanfar
and Xie do not apply; this observation is not an alternative
nonexistence argument.
\end{remark}

\subsection*{Outline of the argument}
The completed cone has a proper regular model $\pi:X\to\Spec R$ with
exceptional divisor $Y$ and conormal bundle $\OO_Y(H)$. Starting with a
finite reflexive module, we extend its bundle on the punctured spectrum
to a reflexive sheaf $\FF$ on $X$. Its restriction $V=\FF|_Y$ is
torsion-free, although it need not be locally free.

Elementary transformations, in the sense related to the constructions
of Chen--Sun \cite{ChenSun2020}, arrange that every normalized
Harder--Narasimhan slope of $V$ belongs to $[3/2,5/2)$. The width-one
window is adapted to the unit slope shift induced by the conormal twist.
Among intervals of width one, centering at $2$, half the normalized
canonical slope, minimizes the largest value of $(s-2)^2$ and bounds it
by $1/4$. Thus the surface Riemann--Roch estimate has coefficient
$2-\tfrac12\cdot\tfrac14=15/8$. This explains the constant without
asserting that no other method could give a stronger bound.
Riemann--Roch, Bogomolov's inequality, and the Hodge index theorem then
give $\chi(V)\leq-c(Y,H)\rk V$. A Koszul estimate bounds the two quantities
\[
 \sum_{j<0}h^0(V(jH))\quad\text{and}\quad
 \sum_{j>0}h^2(V(jH))
\]
by $h^0(V)$ and $h^2(V)$, respectively. These control the total cohomology
that can be lost when one passes from the exceptional divisor to $X$ and
then from $X$ to its complement. The resulting lower bound is
$-\chi(V)$.

The argument uses neither a grading on the module nor a splitting of
the Harder--Narasimhan filtration. In particular, the claimed local
conclusion is not obtained by assuming that all maximal Cohen--Macaulay
modules can be made graded.

\subsection*{Conventions}
All rings are commutative with identity, and all modules and coherent
sheaves under consideration are finite unless explicitly stated
otherwise. For a coherent sheaf $T$ on $Y$, write
$T(j)=T\otimes\OO_Y(jH)$ and $h^i(T)=\dim_\C H^i(Y,T)$.
For a finite-length $R$-module, its length equals its dimension over the
residue field $\C$. Reflexive duals are taken with respect to the
structure sheaf or the ring in question. Linear and numerical equivalence
of divisors are denoted by $\sim$ and $\equiv$, respectively.
Main claim~\ref{thm:main} requires only numerical equivalence, whereas
the canonical-module identification in Consequence~\ref{cor:example}
uses the stated linear equivalence.

\section{The cone, its completion, and reflexive extensions}\label{sec:cone}

Put $L=\OO_Y(H)$, and let
\[
 X_0=\Tot(L^{-1})
     =\Spec_Y\!\left(\bigoplus_{n\geq0}L^n\right).
\]
The zero section will again be denoted by $Y$.

\begin{proposition}\label{prop:cone}
There is a proper morphism $q:X_0\to\Spec B$ which is an isomorphism
away from the zero section and the vertex. Moreover,
\begin{equation}\label{eq:conormal}
 \bb\OO_{X_0}=\OO_{X_0}(-Y),\qquad
 \OO_{X_0}(-Y)|_Y\simeq L.
\end{equation}
The ring $B$ is a finitely generated normal domain of dimension $3$,
and $\OO_{X_0}(-Y)$ is $q$-ample. After the base change
\[
 X=X_0\times_{\Spec B}\Spec R,
\]
one obtains a proper morphism $\pi:X\to\Spec R$ from a regular
three-dimensional scheme, with
\begin{equation}\label{eq:punctured}
 U:=X\setminus Y\simeq\Spec R\setminus\{\mm\},\qquad
 \OO_X(-Y)|_Y\simeq L.
\end{equation}
The scheme-theoretic closed fibre is $Y$, and $\OO_X(-Y)$ is
$\pi$-ample.
\end{proposition}

\begin{proof}
Choose three sections $s_0,s_1,s_2\in H^0(Y,L)$ with no common zero.
Such a choice is possible because $L$ is globally generated and $Y$
has dimension $2$. They define a morphism
$g:Y\to\PP^2$ with $g^*\OO_{\PP^2}(1)\simeq L$.
A curve contracted by $g$ would have degree zero with respect to $L$,
contrary to ampleness. Thus $g$ is finite and surjective. The graded
ring $B$ is consequently finite over
$P=\C[s_0,s_1,s_2]$; this follows as well by applying the standard
finite-generation theorem for twisted sections to $g_*\OO_Y$.
In particular, $\sqrt{(s_0,s_1,s_2)B}=\bb$.

The total space $X_0$ is the base change under $g$ of
$\Tot(\OO_{\PP^2}(-1))=\Bl_0\A^3$. It is therefore finite over that
blow-up and proper over $\Spec P$. The natural evaluation map to
$\Spec B$ is proper: its graph is closed over $\Spec P$, since
$\Spec B$ is separated over $\Spec P$, and the projection of the
graph is a base change of the proper morphism $X_0\to\Spec P$.

For a nonzero section $s\in B_1$, the usual description of a full
section ring gives
\begin{equation}\label{eq:local-cone}
 (B_s)_0=\Gamma(Y_s,\OO_Y),\qquad
 B_s=(B_s)_0[s,s^{-1}].
\end{equation}
For completeness, a function regular on $Y_s$ becomes a global section
of some power of $L$ after multiplication by a sufficiently large
power of $s$; this proves the first equality. The second follows by
using the invertible homogeneous element $s$ of degree one. These
identifications identify $D(s)$ with the complement of the zero
section over $Y_s$. Since the chosen three $D(s_i)$ cover the vertex
complement, they glue to the required isomorphism. This uses the full
section ring and does not assume generation in degree one; compare
\cite[Lemma 1.2]{Anghel2026}.

Locally trivialize $L$ and write $z$ for the fibre coordinate. A section
of $L^n$, viewed as a function on $X_0$, is a multiple of $z^n$.
At every point of $Y$, one of the degree-one sections is a unit times
$z$. Hence the extension of $\bb$ is precisely the ideal $(z)$ of the
zero section, proving \eqref{eq:conormal}.
The invertible sheaf $\OO_{X_0}(-Y)$ is the pullback of
$\OO_{\Bl_0\A^3}(1)$ in the blow-up convention, equivalently the
pullback of $\OO_{\PP^2}(1)$ to that total space. It is relatively
ample over $\Spec P$, and hence over $\Spec B$.

The variety $X_0\setminus Y$ is smooth and integral. Its ring of global
functions is
\[
 \bigoplus_{n\in\Z}H^0(Y,L^n)=B,
\]
since negative powers of an ample line bundle have no sections.
The global functions on a normal integral scheme form an integrally
closed domain. Thus $B$ is normal, and its dimension is $3$ because
it is finite over $P$. The local ring $B_\bb$ is excellent, since it is
essentially of finite type over $\C$; see
\cite[Tags 07QU and 07QW]{Stacks}. In particular it is a G-ring:
its formal fibres are geometrically regular. Thus the flat completion
map $B_\bb\to R$ is regular, as in \cite[Tag 0AH2]{Stacks}.
Normality passes to this completion by \cite[Tag 0C23]{Stacks},
so $R$ is a normal local domain. Write
$X_{0,\bb}=X_0\times_{\Spec B}\Spec B_\bb$.
The morphism $X\to X_{0,\bb}$ is flat with geometrically regular
fibres, by finite-type base change of the regular completion map
\cite[Tag 07C1]{Stacks}. For a point $x\in X$ above $x_0\in X_{0,\bb}$,
the local map $\OO_{X_{0,\bb},x_0}\to\OO_{X,x}$ is flat, its source
is regular, and its closed fibre is regular. Hence $\OO_{X,x}$ is
regular by \cite[Tag 031E]{Stacks}. In particular, over a point
$y\in Y$ that closed fibre is precisely $\kappa(y)$, since completion
does not change the residue field at the vertex. No assertion that
all formal fibres consist of a single point is needed.
Flatness preserves the Cartier divisor in \eqref{eq:conormal}; its fibre is
still $Y$, and its complement is the punctured spectrum in
\eqref{eq:punctured}. Properness and relative ampleness are preserved
by base change. The model has dimension $3$: along $Y$, a local
equation of this divisor is a nonzerodivisor with regular surface
quotient, and outside $Y$ the morphism is an isomorphism.
\end{proof}

\begin{remark}\label{rem:completion}
No descent or algebraization of an $R$-module to a graded $B$-module is
involved in Proposition~\ref{prop:cone}. The model $X$ is an ordinary
scheme obtained by base change, not a formal scheme whose algebraization
is being assumed. This allows a module defined only over the completed
local ring to be pulled back directly to $X$.
\end{remark}

We record the reflexive properties that will be used repeatedly.
The standard equivalence between reflexivity and the combination of
torsion-freeness with Serre's condition $S_2$ on a normal scheme is
given by \cite[Tag 0AY6]{Stacks}; the corresponding module-theoretic
criterion is \cite[Tag 0AVB]{Stacks}.

\begin{lemma}\label{lem:restriction}
Let $X$ be a regular three-dimensional Noetherian scheme and let
$i:Y\hookrightarrow X$ be a smooth integral Cartier surface. If
$\FF$ is a coherent reflexive sheaf on $X$, then $i^*\FF$ is
torsion-free on $Y$.
\end{lemma}

\begin{proof}
A reflexive sheaf on a regular scheme is locally free at points of
codimension at most $2$. Therefore $i^*\FF$ is locally free at the
generic point and at every codimension-one point of $Y$.
At a closed point of $Y$, choose a local equation $z$ for $Y$.
Since $\FF$ is torsion-free, multiplication by $z$ is injective.
The depth of $\FF$ is at least $2$, and hence
\[
 \depth(\FF/z\FF)\geq1.
\]
Thus $i^*\FF$ has no submodule of finite length at that point.
Any torsion subsheaf would have a component of dimension $1$ or $0$;
the preceding local-freeness and depth statements exclude both.
\end{proof}

\begin{lemma}\label{lem:extension}
If $M$ is a finite reflexive $R$-module, then it is locally free on
$\Spec R\setminus\{\mm\}$. The coherent sheaf
\begin{equation}\label{eq:extension}
 \FF=(\pi^*\widetilde M)^{\vee\vee}
\end{equation}
is reflexive on $X$, restricts to $\widetilde M$ on $U$, and has a
torsion-free restriction to $Y$.
\end{lemma}

\begin{proof}
The local rings of the punctured spectrum are regular and have
dimension at most $2$, by Proposition~\ref{prop:cone}. Reflexive
modules over such rings are free. Double dualization on the regular
scheme $X$ gives a coherent reflexive sheaf and does not change the
already locally free restriction over $U$. Apply
Lemma~\ref{lem:restriction} for the last assertion.
\end{proof}

\begin{lemma}[Finite length of the deficiency module]\label{lem:finite-deficiency}
For the ring $R$ of Proposition~\ref{prop:cone} and every finite
reflexive $R$-module $M$, the module $H^2_\mm(M)$ has finite length.
In particular, $M$ is generalized Cohen--Macaulay: all its local
cohomology modules in degrees strictly below $3$ have finite length.
\end{lemma}

\begin{proof}
The assertion is immediate for $M=0$, so assume $M\ne0$.
Choose a presentation $R=S/I$ with
$S=\C[[z_1,\ldots,z_N]]$ a complete regular local ring. Since $R$
is a three-dimensional domain, $I$ is prime of height $N-3$.
Put $D_M=\Ext_S^{N-2}(M,S)$, a finite $R$-module.
For $\mathfrak p\ne\mm$ in $\Spec R$, let $\mathfrak q$ be its inverse
image in $S$. By Lemma~\ref{lem:extension}, $M_\mathfrak p$ is free
of positive rank over the regular local ring $R_\mathfrak p$.
Catenarity and the dimension formula for the regular local ring $S$
give
\[
 \dim S_\mathfrak q-\dim R_\mathfrak p=N-3.
\]
Regularity of $S_\mathfrak q$ gives
$\depth S_\mathfrak q=\dim S_\mathfrak q$ and finite projective dimension
for every finite $S_\mathfrak q$-module. Since $M_\mathfrak p$ is a
nonzero free module over the regular local ring $R_\mathfrak p$, its
depth over $R_\mathfrak p$ equals $\dim R_\mathfrak p$. Moreover, the
surjective local homomorphism $S_\mathfrak q\to R_\mathfrak p$ gives
$\depth_{S_\mathfrak q}M_\mathfrak p=
\depth_{R_\mathfrak p}M_\mathfrak p$.
The Auslander--Buchsbaum formula \cite[Tag 090V]{Stacks} therefore gives
\[
 \operatorname{pd}_{S_\mathfrak q}M_\mathfrak p
 =\depth S_\mathfrak q-\depth_{S_\mathfrak q}M_\mathfrak p
 =\dim S_\mathfrak q-\dim R_\mathfrak p=N-3.
\]
Localization of Ext now yields $(D_M)_\mathfrak p=0$, since $N-2>N-3$.
Thus $D_M$ is supported only at $\mm$ and has finite length.

Let $\mathfrak n_S$ be the maximal ideal of $S$ and let $E_S(\C)$ be
an injective hull of its residue field. A normalized dualizing complex
for $S$ is $\omega_S^\bullet\simeq S[N]$. Applying local duality
\cite[Tag 0AAK]{Stacks} in degree $2$, and using that the finite module
$D_M$ is already complete, gives
\[
 D_M\simeq\Hom_S(H^2_{\mathfrak n_S}(M),E_S(\C))
       =\Hom_S(H^2_\mm(M),E_S(\C)).
\]
The equality uses the fact that the image of $\mathfrak n_S$ under
$S\twoheadrightarrow R$ generates $\mm$. Local cohomology at the maximal
ideal of a finite module is Artinian, so Matlis biduality over the
complete local ring $S$ applies. Since Matlis duality preserves finite
length, $H^2_\mm(M)$ has finite length as well. Finally,
$H^0_\mm(M)=H^1_\mm(M)=0$ because $M$ is reflexive over the normal
local ring $R$ and has depth at least $2$.
\end{proof}

\section{Elementary transformations and slope normalization}\label{sec:normalization}

For a nonzero torsion-free coherent sheaf $T$ on $Y$, set
\[
 d=H^2,\qquad
 s(T)=\frac{c_1(T)\cdot H}{\rk(T)d}.
\]
We call $s(T)$ its normalized slope. Let $\smax(T)$ and $\smin(T)$
be the largest and smallest slopes in its Harder--Narasimhan filtration.
All Harder--Narasimhan filtrations here are taken with respect to $H$
and have torsion-free semistable successive quotients. We use the usual
slope properties, as in \cite[Sections 1.3 and 1.6]{HuybrechtsLehn2010}.
In particular,
\begin{equation}\label{eq:slope-properties}
 s(T(1))=s(T)+1,\qquad
 \smax(T')\leq\max\{\smax(T_1),\smax(T_2)\}
\end{equation}
whenever $0\to T_1\to T'\to T_2\to0$ is an exact sequence of
torsion-free sheaves. The second inequality follows by intersecting an
arbitrary subsheaf of $T'$ with $T_1$ and considering its image in $T_2$.

\begin{lemma}[Elementary transformation]\label{lem:hecke}
Let $\FF$ be reflexive on $X$, put $V=i^*\FF$, and let
$\varphi:V\twoheadrightarrow Q$ be a torsion-free quotient on $Y$ with
kernel $K$. Let $\eta:\FF\to i_*i^*\FF=i_*V$ be the canonical
restriction map, and define
\begin{equation}\label{eq:hecke-definition}
 \FF'=\ker\bigl(\FF\xrightarrow{\eta}i_*V
                      \xrightarrow{i_*\varphi}i_*Q\bigr).
\end{equation}
Then $\FF'$ is reflexive, $\FF'|_U\simeq\FF|_U$, and its restriction
fits into the exact sequence
\begin{equation}\label{eq:hecke-restriction}
 0\longrightarrow Q(1)\longrightarrow i^*\FF'
   \longrightarrow K\longrightarrow0.
\end{equation}
In particular,
\begin{equation}\label{eq:degree-increase}
 c_1(i^*\FF')\cdot H=c_1(V)\cdot H+\rk(Q)d.
\end{equation}
\end{lemma}

\begin{proof}
The map $\eta$ is surjective: locally, for a defining equation $z$ of
$Y$, it is the quotient map $\FF\to\FF/z\FF$. This surjectivity does
not require Tor-vanishing. Since $i$ is a closed immersion, its direct
image on sheaves of modules is exact; hence $i_*\varphi$ is also
surjective. Their composite therefore gives a short exact sequence
\[
 0\longrightarrow\FF'\longrightarrow\FF
   \xrightarrow{(i_*\varphi)\circ\eta}i_*Q\longrightarrow0.
\]
The kernel is coherent and torsion-free, and the quotient is supported
on $Y$, so the restrictions to $U$ are the same. We verify $S_2$ for
$\FF'$.
At a closed point of $Y$, the torsion-free surface sheaf $Q$ has depth
at least $1$, while $\FF$ has depth at least $2$. The depth lemma
therefore gives depth at least $2$ for $\FF'$. At a codimension-one
point of $Y$, both $V$ and $Q$ are finite torsion-free modules over the
corresponding discrete valuation ring, and hence are free. The same depth lemma over the regular two-dimensional
local ring of $X$ again gives depth at least $2$. At the generic point
of $Y$, torsion-freeness gives depth at least $1$. Away from $Y$ there
is nothing to check. Thus $\FF'$ is $S_2$ and torsion-free, hence
reflexive.

Since a local equation of $Y$ is a nonzerodivisor on $\FF$,
$\Tor_1^{\OO_X}(\FF,\OO_Y)=0$. On the other hand, the two-term
locally free resolution of $\OO_Y$ gives
\[
 \Tor_1^{\OO_X}(i_*Q,\OO_Y)
       \simeq Q\otimes\OO_X(-Y)|_Y\simeq Q(1).
\]
Tensoring $0\to\FF'\to\FF\to i_*Q\to0$ with $\OO_Y$ now yields
\eqref{eq:hecke-restriction}. Additivity of first Chern classes gives
\eqref{eq:degree-increase}.
\end{proof}

This is the kernel form of the reflexive Hecke transformation;
compare \cite[Lemmas 2.2 and 2.4 and Proposition 2.6]{ChenSun2020}.
The proof above is given in the algebraic setting needed here and allows
$Q$ to fail to be locally free at finitely many points.

\begin{remark}\label{rem:sign}
The twist in \eqref{eq:hecke-restriction} is the conormal bundle, not
the normal bundle. The case $\FF=\OO_X$ and $Q=\OO_Y$ gives
$\FF'=\OO_X(-Y)$ and $i^*\FF'=L$, which also fixes the sign directly.
\end{remark}

\begin{proposition}[Fixed-interval normalization]\label{prop:normalization}
Let $\FF$ be a nonzero coherent reflexive sheaf on $X$ and let
$V=i^*\FF$. By a finite sequence of twists by multiples of $Y$ and
elementary transformations of the form \eqref{eq:hecke-definition},
one can replace $\FF$ by a reflexive sheaf $\GG$ such that
$\GG|_U\simeq\FF|_U$ and
\begin{equation}\label{eq:normalized-interval}
 \frac32\leq\smin(i^*\GG)\leq\smax(i^*\GG)<\frac52.
\end{equation}
\end{proposition}

\begin{proof}
Write $r=\rk\FF$. Since
$i^*(\FF(aY))\simeq V(-a)$, a twist by a sufficiently large integer
$a$ makes $\smax(V)<5/2$. This twist is trivial on $U$.
Suppose the resulting $V$ still has $\smin(V)<3/2$. Take the last
Harder--Narasimhan quotient $Q$, so $Q$ is semistable of slope
$\smin(V)$, and write $K$ for the preceding step of the filtration.
If $V$ is semistable, this notation allows $Q=V$ and $K=0$.
In that case the transformation is simply
$\FF'=\ker(\FF\to i_*V)=\FF(-Y)$, because the local equation of $Y$
is a nonzerodivisor on $\FF$, and its restriction is $V(1)$.

Perform the transformation of Lemma~\ref{lem:hecke}. Its restriction
$V'$ is an extension of $K$ by $Q(1)$. The latter is semistable of
slope $\smin(V)+1<5/2$, while every Harder--Narasimhan slope of $K$
is also less than $5/2$. Equation~\eqref{eq:slope-properties} implies
$\smax(V')<5/2$.

For an explicit termination bound, denote the restriction after the
initial twist by $V_0$, and after $n$ further transformations by $V_n$.
Put $e_n=c_1(V_n)\cdot H$, and let $q_n\geq1$ be the rank of the last
Harder--Narasimhan quotient used at step $n$. The rank remains $r$, and
\[
 e_{n+1}=e_n+q_nd,\qquad e_n<\frac52rd.
\]
Consequently the integers
\begin{equation}\label{eq:normalization-potential}
 D_n=\left\lceil\frac{5r}{2}-\frac{e_n}{d}\right\rceil-1
 \quad\text{satisfy}\quad
 D_n\geq0,\qquad D_{n+1}=D_n-q_n\leq D_n-1.
\end{equation}
Thus at most $D_0$ further transformations are possible. At termination
the minimum slope is at least $3/2$, since otherwise the construction
would supply another step. The maximum slope remains less than $5/2$,
giving \eqref{eq:normalized-interval}. Lemma~\ref{lem:hecke} ensures
reflexivity and invariance on $U$ at every step. This bound depends on
the initial sheaf and the chosen first twist; no uniform bound over all
modules is asserted.
\end{proof}

\begin{remark}\label{rem:no-splitting}
The terminating quantity in \eqref{eq:normalization-potential} is an
integer function of the total degree, not the minimum slope.
The proof does not require the minimum slope to increase strictly at
each step, does not permute Harder--Narasimhan factors, and does not
split any extension. This is why a filtration by line bundles is not
needed.
\end{remark}

\section{Cohomological estimates on the exceptional surface}\label{sec:surface-estimates}

We now assume $K_Y\equiv4H$. The next two estimates concern arbitrary
torsion-free sheaves on $Y$ with slopes in the interval
\eqref{eq:normalized-interval}.

\begin{proposition}[Euler characteristic]\label{prop:euler}
Let $V$ be a torsion-free sheaf of rank $r>0$ satisfying
\eqref{eq:normalized-interval}. Then
\begin{equation}\label{eq:euler-bound}
 \chi(V)\leq
 \left(\chi(\OO_Y)-\frac{15d}{8}\right)r=-c(Y,H)r.
\end{equation}
\end{proposition}

\begin{proof}
Take a semistable Harder--Narasimhan quotient $Q$ of $V$ and put
$q=\rk Q$, $\nu=c_1(Q)/q$, and $s=\nu\cdot H/d$.
Bogomolov's inequality in characteristic zero gives
\[
 \Delta(Q):=2q\,c_2(Q)-(q-1)c_1(Q)^2\geq0.
\]
For the surface form, see \cite[Theorem 4.4]{Langer2022}, which applies
to torsion-free sheaves. That statement uses a very ample divisor:
choose $m>0$ such that $mH$ is very ample. Multiplication of the
polarization by $m$ multiplies all slopes by the same positive number,
so $H$-semistability and $mH$-semistability coincide. The conclusion
$\Delta(Q)\geq0$ is independent of this auxiliary choice; the divisor
$H$, the line bundle $L$, and the section ring in our construction
are unchanged. The Hodge index theorem gives $\nu^2\leq s^2d$,
and numerical equivalence gives $\nu\cdot K_Y=4\nu\cdot H$.
Riemann--Roch on the smooth surface yields
\begin{align}
 \frac{\chi(Q)}{q}
 &=\chi(\OO_Y)+\frac{\nu^2-\nu\cdot K_Y}{2}
              -\frac{\Delta(Q)}{2q^2}\notag\\
 &\leq\chi(\OO_Y)+\frac d2(s^2-4s)\notag\\
 &=\chi(\OO_Y)-2d+\frac d2(s-2)^2\notag\\
 &\leq\chi(\OO_Y)-\frac{15d}{8}.
 \label{eq:RR}
\end{align}
The last inequality uses $3/2\leq s<5/2$. Additivity of Euler
characteristic along the Harder--Narasimhan filtration proves the
claim. To justify the Riemann--Roch identity for a possibly non-locally-free
$Q$, take a finite locally free resolution on the smooth projective
surface. Apply Hirzebruch--Riemann--Roch to its locally free terms
\cite[Appendix A, Theorem 4.1]{Hartshorne1977}, and use additivity of
Euler characteristic and of the Chern character. This gives the displayed
identity for $Q$ without invoking the relative Grothendieck--Riemann--Roch
theorem. For the Hodge index theorem, see
\cite[Chapter V, Theorem 1.9]{Hartshorne1977}.
\end{proof}

\begin{lemma}[Two-section Koszul estimate]\label{lem:koszul}
Let $T$ be a torsion-free coherent sheaf on $Y$. If $H^0(Y,T(-3))=0$,
then
\begin{equation}\label{eq:koszul-bound}
 h^0(T)\geq h^0(T(-1))+h^0(T(-2)).
\end{equation}
\end{lemma}

\begin{proof}
The morphism defined by $|H|$ is finite onto a surface: a contracted
curve would have $H$-degree zero, contradicting ampleness. Thus
$h^0(Y,L)\geq3$. Two general hyperplanes meet its image in a finite
set, and their inverse images have no common curve.
The non-locally-free locus of $T$ is finite. By global generation,
we can choose two such sections $a,b$ of $L$ whose common zero locus
also avoids that finite set. At a common zero, $T$ is free and $a,b$
are a regular sequence. Outside the common zero locus of $a$ and $b$,
at least one of them is a unit in a local trivialization of $L$.
Consequently, the Koszul complex tensored with $T$ is exact in its
first two positions:
\[
 0\longrightarrow T(-2)\xrightarrow{(-b,a)}T(-1)^{\oplus2}
   \xrightarrow{(a,b)}T.
\]
Taking global sections gives, with $a_j=h^0(T(-j))$,
\[
 a_0\geq2a_1-a_2.
\]
Apply the same argument after twisting by $L^{-1}$ to get
$a_1\geq2a_2-a_3=2a_2$. Hence
\[
 a_0-a_1-a_2
   =(a_0-2a_1+a_2)+(a_1-2a_2)\geq0.
\]
This proves \eqref{eq:koszul-bound}.
\end{proof}

\begin{proposition}[The two tails]\label{prop:tails}
Let $V$ satisfy \eqref{eq:normalized-interval}. Define
\begin{equation}\label{eq:tails}
 T_-(V)=\sum_{j<0}h^0(V(j)),\qquad
 T_+(V)=\sum_{j>0}h^2(V(j)).
\end{equation}
These sums are finite, and
\begin{equation}\label{eq:tail-bounds}
 T_-(V)\leq h^0(V),\qquad T_+(V)\leq h^2(V).
\end{equation}
More precisely, only $j=-1,-2$ can contribute to the first sum and
only $j=1,2$ can contribute to the second.
\end{proposition}

\begin{proof}
A nonzero section of a torsion-free sheaf gives an injection of
$\OO_Y$, so a sheaf with negative maximal slope has no sections.
For $j\geq3$, the maximal slope of $V(-j)$ is less than $5/2-j<0$.
Thus $H^0(V(-j))=0$ for all such $j$, and
Lemma~\ref{lem:koszul} gives the first inequality.

Let $V^D=V^\vee\otimes\omega_Y$.
On the smooth surface $Y$, both $V^{\vee\vee}$ and $V^\vee$ are
locally free, and $V^\vee\simeq(V^{\vee\vee})^\vee$.
The quotient $V^{\vee\vee}/V$ is zero-dimensional, so passage to the
reflexive hull does not change the Harder--Narasimhan slopes.
Dualizing the locally free hull reverses and negates these slopes, so
$\smax(V^\vee)=-\smin(V^{\vee\vee})=-\smin(V)$;
see \cite[Sections 1.1 and 1.6]{HuybrechtsLehn2010}.
Numerical equivalence $K_Y\equiv4H$ gives
$s(\omega_Y)=K_Y\cdot H/d=4$, and hence
\[
 \smax(V^D)=4-\smin(V)\leq\frac52.
\]
Hence $H^0(V^D(-j))=0$ for $j\geq3$.
Since $Y$ is smooth and projective, its canonical sheaf $\omega_Y$
is a dualizing sheaf \cite[Chapter III, Corollary 7.12]{Hartshorne1977}.
The top-cohomology duality for an arbitrary coherent sheaf, namely
the degree-zero case of \cite[Chapter III, Theorem 7.6]{Hartshorne1977},
therefore identifies
\[
 H^2(Y,V(j))^*\simeq
 \Hom_Y(V(j),\omega_Y)=H^0(Y,V^D(-j)).
\]
Apply Lemma~\ref{lem:koszul} to $V^D$ to obtain
\[
 T_+(V)=h^0(V^D(-1))+h^0(V^D(-2))
       \leq h^0(V^D)=h^2(V).
\]
The local freeness of $V^\vee$ used above is automatic on the smooth
surface; no local-freeness assumption on $V$ itself has been imposed.
\end{proof}

\section{Persistence on the punctured spectrum}\label{sec:persistence}

The next estimate keeps track of arbitrary extension data in the
normal direction to $Y$. In particular, none of the connecting maps
in its proof is assumed to be zero.

\begin{proposition}[Fixed-source persistence]\label{prop:persistence}
Let $Y$ be any smooth connected complex projective surface with an
ample globally generated divisor $H$, without imposing
\eqref{eq:main-hypotheses}. Let $\pi:X\to\Spec R$ be the associated
model of Proposition~\ref{prop:cone}, whose construction uses only
these geometric assumptions, and let $\FF$ be a coherent reflexive
sheaf on $X$. Put $V=\FF|_Y$, and assume that the two sums in
\eqref{eq:tails} are finite. Then
\begin{equation}\label{eq:persistence}
 \dim_\C H^1(U,\FF|_U)
 \geq h^1(V)-T_+(V)-T_-(V).
\end{equation}
Specifically, the image of $H^1(X,\FF)$ in this cohomology has
length at least the right-hand side.
\end{proposition}

We isolate the localization step before proving the proposition.
All restriction maps below use the canonical trivializations of the
bundles $\OO_X(nY)$ on $U$.

\begin{lemma}[Clearing denominators]\label{lem:clearing-denominators}
Let $X$ be a quasi-compact separated scheme, let $Y\subset X$ be an
effective Cartier divisor, and put $U=X\setminus Y$. Let $\FF$ be a
quasi-coherent sheaf on which local equations of $Y$ act injectively.
If $a\geq0$ is an integer and a class $\xi\in H^1(X,\FF(aY))$ restricts to zero
on $U$, then its image in $H^1(X,\FF(tY))$ is zero for some
integer $t\geq a$.

In particular, put $V=\FF|_Y$ and $L=\OO_X(-Y)|_Y$. If
$H^0(Y,V\otimes L^{-n})=0$ for every $n>a$, then the restriction map
\[
 \gamma_a:H^1(X,\FF(aY))\longrightarrow H^1(U,\FF|_U)
\]
is injective.
\end{lemma}

\begin{proof}
Choose a finite affine cover $X=\bigcup_{i=1}^m W_i$ trivializing
$\OO_X(Y)$, write $W_i=\Spec A_i$, and let $z_i\in A_i$ be a local
equation of $Y$. Set $M_i=\Gamma(W_i,\FF)$. Then
$W_i\cap U=D(z_i)$ and
$\Gamma(W_i\cap U,\FF|_U)=(M_i)_{z_i}$.
Since $z_i$ acts injectively, the sections of $\FF(tY)$ on $W_i$
identify with $z_i^{-t}M_i\subset(M_i)_{z_i}$ under the canonical
trivialization on $U$; the transition maps are inclusions.

The finite intersections of the $W_i$ are affine, as are their
intersections with $U$. Hence these covers compute quasi-coherent
cohomology by \v{C}ech complexes
\cite[Tags 0BDX and 01XD]{Stacks}.
Represent $\xi$ by a $1$-cocycle $(c_{ij})$. Its vanishing on $U$
means that there are $b_i\in(M_i)_{z_i}$ with
$c_{ij}|_U=b_j-b_i$ on overlaps. Each $b_i$ has a finite denominator
$z_i^{n_i}$. The cover is finite, so choose one
$t\geq\max\{a,n_1,\ldots,n_m\}$. Then the $b_i$ lift to sections
$\widetilde b_i\in\Gamma(W_i,\FF(tY))$.

On $W_i\cap W_j$, the image of $c_{ij}$ minus
$\widetilde b_j-\widetilde b_i$ vanishes after inverting $z_i$.
It is therefore killed by a power of $z_i$, which acts injectively
on the sections of $\FF(tY)$. The difference is zero already on
$W_i\cap W_j$. Thus the image cocycle is a coboundary at the finite
stage $t$, proving the first assertion.

For the second assertion, let $i:Y\hookrightarrow X$ be the inclusion.
For every $s\geq a$, multiplication by the Cartier equation gives
\[
 0\longrightarrow\FF(sY)\longrightarrow\FF((s+1)Y)
 \longrightarrow i_*(V\otimes L^{-(s+1)})\longrightarrow0.
\]
The last term has no global sections, so
$H^1(X,\FF(sY))\to H^1(X,\FF((s+1)Y))$ is injective.
A class in $\ker\gamma_a$ dies at a finite later stage by the first
assertion; since all intervening maps are injective, it is zero.
\end{proof}

\begin{proof}[Proof of Proposition~\ref{prop:persistence}]
For $j\in\Z$, set $\FF_j=\FF(-jY)$. The local equation of $Y$ is
a nonzerodivisor on $\FF$, so there are exact sequences
\begin{equation}\label{eq:restriction-sequence}
 0\longrightarrow\FF_{j+1}\longrightarrow\FF_j
  \longrightarrow i_*V(j)\longrightarrow0.
\end{equation}
For $p>0$, each $H^p(X,\FF_j)$ is a finite $R$-module supported at
$\mm$: properness gives finiteness, and the higher direct images
vanish off $\mm$ because $\pi$ is an isomorphism there. These
cohomology modules therefore have finite length. Since $Y$ is the
scheme-theoretic closed fibre, $\mm$ annihilates $H^p(Y,V(j))$;
its $R$-length is therefore $h^p(V(j))$.

The base $\Spec R$ is affine, so the $\pi$-ample invertible sheaf
$\OO_X(-Y)$ is ample on $X$ by \cite[Tag 01VK]{Stacks}.
Since $\pi$ is proper and $R$ is Noetherian, Serre vanishing gives
$H^p(X,\FF_j)=0$ for every $p>0$ and all sufficiently large $j$;
see \cite[Tag 0B5U]{Stacks}. Choose $J\geq2$ with
$H^2(X,\FF_J)=0$.
For every $j\geq1$, the long exact sequence of
\eqref{eq:restriction-sequence} contains
\[
 H^2(X,\FF_{j+1})\longrightarrow H^2(X,\FF_j)
       \longrightarrow H^2(Y,V(j)).
\]
Exactness at the middle term gives
\[
 \len_R H^2(X,\FF_j)
 \leq\len_R H^2(X,\FF_{j+1})+h^2(V(j)).
\]
No vanishing of $H^3(X,\FF_{j+1})$ is used. Applying this inequality
successively for $j=J-1,J-2,\ldots,1$ yields
\begin{equation}\label{eq:positive-loss}
 \len_R H^2(X,\FF_1)
 \leq\sum_{j=1}^{J-1}h^2(V(j))
 \leq\sum_{j\geq1}h^2(V(j))=T_+(V).
\end{equation}
The case $j=0$ of \eqref{eq:restriction-sequence} contains
\[
 H^1(X,\FF)\longrightarrow H^1(Y,V)
       \longrightarrow H^2(X,\FF_1).
\]
Write $A_1=H^1(X,\FF)$ and let $\delta$ denote the displayed connecting
map. Exactness and \eqref{eq:positive-loss} give
\begin{align}
 \len_R A_1
 &\geq\len_R\im\bigl(A_1\to H^1(Y,V)\bigr)\notag\\
 &=h^1(V)-\len_R\im\delta
 \geq h^1(V)-T_+(V).
 \label{eq:fixed-source-length}
\end{align}

For $s\geq0$, let $\QQ_s=\FF(sY)/\FF$. It has a finite filtration
with factors $i_*V(j)$ for $-s\leq j\leq-1$. The elementary
subadditivity of the dimension of global sections therefore gives
\[
 \len_R H^0(X,\QQ_s)
       \leq\sum_{j=-s}^{-1}h^0(V(j))\leq T_-(V).
\]
In the cohomology sequence of
$0\to\FF\to\FF(sY)\to\QQ_s\to0$, the kernel
\begin{equation}\label{eq:fixed-kernel}
 K_s=\ker\bigl(A_1\longrightarrow H^1(X,\FF(sY))\bigr)
\end{equation}
is an image of $H^0(X,\QQ_s)$. Thus
\begin{equation}\label{eq:negative-loss}
 \len_R K_s\leq T_-(V)\qquad\text{for every }s\geq0.
\end{equation}

Since $T_-(V)$ is finite and its summands are nonnegative integers,
choose $a\geq0$ such that $H^0(Y,V(-n))=0$ for every $n>a$.
The scheme $X$ is quasi-compact and separated, since it is proper
over the affine scheme $\Spec R$, and a local equation of $Y$
acts injectively on $\FF$. Lemma~\ref{lem:clearing-denominators}
therefore shows that
\[
 \gamma_a:H^1(X,\FF(aY))\longrightarrow H^1(U,\FF|_U)
\]
is injective. Let $\beta_a:A_1\to H^1(X,\FF(aY))$ be the natural
map and let $\alpha:A_1\to H^1(U,\FF|_U)$ be restriction. The
canonical trivializations on $U$ give $\alpha=\gamma_a\circ\beta_a$;
hence $\ker\alpha=\ker\beta_a=K_a$. Using
\eqref{eq:fixed-source-length} and \eqref{eq:negative-loss}, we obtain
\[
 \len_R\im\alpha=\len_R A_1-\len_R K_a
 \geq h^1(V)-T_+(V)-T_-(V).
\]
The image is a finite-length submodule of $H^1(U,\FF|_U)$, and its
length equals its $\C$-dimension. This proves \eqref{eq:persistence}.
\end{proof}

\begin{remark}\label{rem:fixed-source}
The standard cohomology-localization formula is
\begin{equation}\label{eq:cohomology-colimit}
 H^1(U,\FF|_U)\simeq\varinjlim_{s\geq0}H^1(X,\FF(sY));
\end{equation}
see \cite[Tag 09MR]{Stacks}. Lemma~\ref{lem:clearing-denominators}
proves its kernel assertion directly, without assuming this formula.
With the cutoff $a$ in the proof, $\ker\alpha=\bigcup_{s\geq0}K_s=K_a$.
Thus \eqref{eq:negative-loss} bounds the entire kernel from the fixed
source $A_1$, not merely the kernels of successive transition maps.
No injectivity of the initial transition maps or surjectivity of
$\gamma_a$ is asserted.
\end{remark}

\begin{corollary}[Finite-stage persistence]\label{cor:finite-stage}
In the setting of Proposition~\ref{prop:persistence}, suppose more
specifically that
\[
 H^0(Y,V(-j))=H^2(Y,V(j))=0\qquad(j\geq3).
\]
Then $H^2(X,\FF(-3Y))=0$, and the natural map
\begin{equation}\label{eq:finite-stage-injection}
 \gamma:H^1(X,\FF(2Y))\longrightarrow H^1(U,\FF|_U)
\end{equation}
is injective. If
$\beta:H^1(X,\FF)\to H^1(X,\FF(2Y))$ is the natural map and
$\alpha=\gamma\circ\beta$, then $\ker\alpha=\ker\beta$ and
\begin{equation}\label{eq:finite-stage-bound}
 \len_R\im\alpha=\len_R\im\beta
 \geq h^1(V)-\sum_{j=1}^{2}\bigl(h^2(V(j))+h^0(V(-j))\bigr).
\end{equation}
\end{corollary}

\begin{proof}
Use $\FF_j=\FF(-jY)$ as in \eqref{eq:restriction-sequence}.
For $j\geq3$, the positive-twist vanishing makes
$H^2(X,\FF_{j+1})\to H^2(X,\FF_j)$ surjective. Serre vanishing
starts the descent at arbitrarily large $j$, so $H^2(X,\FF_3)=0$.
Descending twice more gives
\[
 \len_R H^2(X,\FF_1)\leq h^2(V(1))+h^2(V(2)).
\]
The negative-twist vanishing allows $a=2$ in
Lemma~\ref{lem:clearing-denominators}. Thus $\gamma=\gamma_2$ is
injective, and consequently $\ker\alpha=\ker\beta$.
The quotient $\FF(2Y)/\FF$ has successive factors $i_*V(-1)$ and
$i_*V(-2)$, so
\[
 \len_R\ker\beta\leq h^0(V(-1))+h^0(V(-2)).
\]
Finally the $j=0$ restriction sequence gives
$\len_R H^1(X,\FF)\geq h^1(V)-\len_R H^2(X,\FF_1)$.
Subtracting the kernel bound yields \eqref{eq:finite-stage-bound}.
Only the kernel from the fixed source stabilizes at this explicit
stage; no surjectivity of $\gamma$ is asserted.
\end{proof}

\section{Proposed proof of the main claim}\label{sec:main-proof}

\begin{proof}[Proposed proof of Main claim~\ref{thm:main}]
Proposition~\ref{prop:cone} gives the stated properties of $R$ and
the model $\pi:X\to\Spec R$. Let $M\ne0$ be a finite reflexive
$R$-module of rank $r$. Since $R$ is a domain and $M$ is torsion-free,
$r>0$. Lemma~\ref{lem:extension} supplies a reflexive extension of
its bundle on the punctured spectrum. Apply
Proposition~\ref{prop:normalization}, and denote the resulting
extension again by $\FF$. All changes take place on the fixed model:
$\pi$, $Y$, $L$, and $U$ are unchanged. The sheaf $\FF$ remains
reflexive, so its actual restriction $V=i^*\FF$ is torsion-free by
Lemma~\ref{lem:restriction}, has rank $r$, and satisfies
\eqref{eq:normalized-interval}. No reflexive hull is substituted for
$V$ in the following cohomology estimates.

Proposition~\ref{prop:tails} verifies the vanishing hypotheses of
Corollary~\ref{cor:finite-stage}. With $\alpha,\beta$ as in that corollary,
its finite-stage form of persistence and
Propositions~\ref{prop:euler} and \ref{prop:tails} give
\begin{align*}
 \dim_\C H^1(U,\FF|_U)
 &\geq\len_R\im\alpha=\len_R\im\beta\\
 &\geq h^1(V)-T_+(V)-T_-(V)\\
 &\geq h^1(V)-h^2(V)-h^0(V)\\
 &=-\chi(V)\\
 &\geq c(Y,H)r.
\end{align*}
Every transformation preserved $\FF|_U\simeq\widetilde M|_U$.
The localization sequence for local cohomology on the affine scheme
$\Spec R$ gives
\begin{equation}\label{eq:local-cohomology-identification}
 H^1(U,\widetilde M|_U)\simeq H^2_\mm(M).
\end{equation}
This proves \eqref{eq:main-bound}.

It remains to check that the bound excludes all finite maximal
Cohen--Macaulay modules, not just those initially assumed reflexive.
Suppose $M$ is a nonzero finite module with $\depth_R M=3$.
Then $M$ is Cohen--Macaulay of dimension $3$. For every
$\mathfrak p\in\Ass_R M$, the quotient $R/\mathfrak p$ has
dimension $3$ by \cite[Tag 0BUS]{Stacks}.
Since $R$ is a three-dimensional local domain, this
forces $\mathfrak p=(0)$; a nonzero prime would add a nontrivial step
below any chain to the maximal ideal. Thus $M$ is torsion-free and
has full support. Cohen--Macaulayness localizes
\cite[Tag 0AAG]{Stacks}, so $M$ satisfies $S_2$.
It is therefore reflexive over the normal ring $R$ by
\cite[Tag 0AVB]{Stacks}.

The bound already proved applies to this $M$ and gives
$H^2_\mm(M)\ne0$. This contradicts $\depth_R M=3$. Hence no such
module exists.
\end{proof}

\begin{remark}\label{rem:scope}
The characteristic-zero input enters through Bogomolov's inequality
for semistable sheaves. The proof is not a positive-characteristic
argument with the Frobenius step omitted. It also makes no assertion
that a given local module can be made graded. The bundle on $U$ stays
fixed, but its reflexive extension across $Y$ is deliberately changed.
\end{remark}

\section{An explicit six-line surface}\label{sec:example}

The surface used here is the degree-six member of the six-line
Hirzebruch--Kummer construction. Hirzebruch's line-arrangement surfaces
are discussed in \cite[Section 2]{Hirzebruch1983}; the polarization relevant here
appears in \cite[Remark 4.2 and Corollary 4.3]{Anghel2026}.
We give equations and computations so that the application of
Main claim~\ref{thm:main} does not depend on a graded nonexistence
theorem. In plane coordinates $[x:y:z]$, the arrangement is
\[
 x\,y\,z\,(x-y)(x-z)(y-z)=0.
\]
These are the six lines joining the pairs of the four points
$(1:0:0)$, $(0:1:0)$, $(0:0:1)$, and $(1:1:1)$; they have four triple
points and three double points, the complete-quadrilateral arrangement
of \cite[Section 1.1]{Hirzebruch1983}.
The map from \eqref{eq:W} to this plane is
$[u_0:\cdots:u_5]\mapsto[u_0^6:u_1^6:u_2^6]$.
Its remaining three coordinates are sixth roots of $x-y$, $x-z$, and
$y-z$. Thus the resolved polarized surface is the $n=6$ member of the
second family in \cite[Remark 4.2]{Anghel2026}.

\subsection{The complete intersection and its singular points}
Let $W$ be the surface in \eqref{eq:W}. Its homogeneous coordinate
ring is free of rank $6^3$ over $\C[u_0,u_1,u_2]$, with basis the
monomials $u_3^a u_4^b u_5^c$, $0\leq a,b,c<6$. The equations,
being successive monic equations in different variables, form a
regular sequence. Thus $W$ is a complete intersection of type
$(6,6,6)$.

This coordinate ring is a domain. Indeed, over
$K=\C(u_0,u_1,u_2)$ one adjoins sixth roots of
\[
 u_0^6-u_1^6,\qquad u_0^6-u_2^6,\qquad u_1^6-u_2^6.
\]
The three functions have disjoint sets of irreducible linear factors.
Taking the valuation at one factor of each shows that a product of
their powers can be a sixth power only if each exponent is divisible
by $6$. Kummer theory therefore gives a field extension of degree
$6^3$. The free coordinate ring injects into its generic fibre, which
is this field, proving the assertion. In particular, $W$ is integral.

\begin{lemma}\label{lem:singular-points}
The surface $W$ has exactly $144$ singular points. Each has a completed
local ring isomorphic to
\[
 \C[[x,y,z]]/(x^6+y^6+z^6).
\]
The blow-up of these points is a smooth projective surface $Y$. Its
exceptional curves are pairwise disjoint smooth plane sextics of
self-intersection $-6$.
\end{lemma}

\begin{proof}
Put $v_i=u_i^6$. The coefficient matrix of the three equations in the
$v_i$ is
\begin{equation}\label{eq:matrix}
 C_0=\begin{pmatrix}
 -1&1&0&1&0&0\\
 -1&0&1&0&1&0\\
 0&-1&1&0&0&1
 \end{pmatrix}.
\end{equation}
For a point with active coordinate set $J=\{i:u_i\ne0\}$, the
Jacobian has the rank of the submatrix $(C_0)_J$. The point lies on
$W$ precisely when the kernel of that submatrix contains the vector
of nonzero sixth powers. No two columns of $C_0$ are proportional.
The dependent triples of columns are exactly
\begin{equation}\label{eq:singular-supports}
 \{0,1,2\},\quad\{0,3,4\},\quad
 \{1,3,5\},\quad\{2,4,5\}.
\end{equation}
Their one-dimensional kernels have no zero coordinate. There is no
rank-two set of four or more columns: such a set would contain two
dependent triples sharing two indices, whereas the four triples in
\eqref{eq:singular-supports} do not. These observations exhaust the
rank-deficient supports.

Each support in \eqref{eq:singular-supports} fixes one projective vector
of sixth powers. Taking sixth roots produces $6^2=36$ projective
points for each support, hence $144$ points in total. At each point,
the Jacobian has rank $2$. On an affine chart in which one active
coordinate is fixed, two active coordinates can be eliminated by the
formal implicit function theorem. A linear combination of the three
original equations that kills the active columns leaves a sum of the
sixth powers of the three inactive coordinates, with all coefficients
nonzero. After rescaling these three coordinates, the resulting
completed local equation is $x^6+y^6+z^6$.

A cone over a smooth plane curve is resolved at its vertex by blowing
up the vertex. The exceptional curve is the plane curve itself, with
normal bundle $\OO_C(-1)$. Here the curve is a smooth plane sextic,
so its self-intersection is $-\deg\OO_C(1)=-6$.
These local resolutions at the isolated points glue to the stated
blow-up. The complete intersection is $S_2$ and regular in codimension
one, so it is normal. The resulting surface $Y$ is smooth, projective,
and connected.
\end{proof}

For reference, the residual equations in the proof can be read from
the following exact table. Coordinates in the third column are those
listed in the second column, in the displayed order; nonzero scalar
multiples of a row give the same equation.
\begin{center}
\begin{tabular}{ccc}
\toprule
Active coordinates & Inactive coordinates & Residual coefficients\\
\midrule
$\{0,1,2\}$ & $(3,4,5)$ & $(1,-1,1)$\\
$\{0,3,4\}$ & $(1,2,5)$ & $(-1,1,1)$\\
$\{1,3,5\}$ & $(0,2,4)$ & $(-1,1,1)$\\
$\{2,4,5\}$ & $(0,1,3)$ & $(-1,1,1)$\\
\bottomrule
\end{tabular}
\end{center}

\subsection{Intersection numbers and the canonical divisor}
Write $E=\sum_{i=1}^{144}E_i$ and $A=\rho^*\OO_W(1)$.
The complete-intersection degree and Lemma~\ref{lem:singular-points}
give
\begin{equation}\label{eq:intersection-data}
 A^2=216,\qquad A\cdot E=0,\qquad E^2=-864.
\end{equation}
Therefore, for $H=3A-E$,
\begin{equation}\label{eq:H-square}
 H^2=9\cdot216-864=1080.
\end{equation}

Adjunction for the complete intersection gives
$\omega_W\simeq\OO_W(12)$.
The discrepancy of the blow-up of a degree-six hypersurface cone in
three affine variables is $2-6=-4$: the ambient threefold blow-up
contributes $2E$, and the strict transform of the hypersurface
subtracts $6E$. Applying adjunction on this blow-up gives
\begin{equation}\label{eq:canonical-example}
 K_Y\sim12A-4E=4H.
\end{equation}
This is an equality up to linear equivalence of integral divisors.
In particular, $K_Y^2=16\cdot1080=17280$.

\subsection{The Euler characteristic}
The Koszul resolution of the complete intersection is
\[
 0\longrightarrow\OO_{\PP^5}(-18)
 \longrightarrow\OO_{\PP^5}(-12)^{\oplus3}
 \longrightarrow\OO_{\PP^5}(-6)^{\oplus3}
 \longrightarrow\OO_{\PP^5}
 \longrightarrow\OO_W\longrightarrow0.
\]
Using $\chi(\OO_{\PP^5}(t))=\binom{t+5}{5}$, where the binomial is
interpreted as a polynomial in $t$, one obtains
\begin{equation}\label{eq:chiW}
 \chi(\OO_W)
   =1-3(-1)+3(-462)-(-6188)=4806.
\end{equation}

Let $C$ be a smooth plane sextic, and let
$Z_C=\Tot(\OO_C(-1))$ be the resolution of its affine cone, with affine
projection $p:Z_C\to C$. The canonical bundle of $C$ is
$\omega_C\simeq\OO_C(3)$. The algebra of the total space is split:
\[
 p_*\OO_{Z_C}=\bigoplus_{j\geq0}\OO_C(j).
\]
More explicitly, if $\mathcal J$ is the ideal of the zero section and
$C_n$ is defined by $\mathcal J^{n+1}$, then the induced affine map
$p_n:C_n\to C$ satisfies
\[
 (p_n)_*\OO_{C_n}\simeq\bigoplus_{j=0}^{n}\OO_C(j).
\]
This is the truncated symmetric algebra, and the transition to $C_{n-1}$
is the projection deleting its degree-$n$ summand. In particular,
$H^0(C_n,\OO_{C_n})\to H^0(C_{n-1},\OO_{C_{n-1}})$ is surjective,
and
$H^1(C_n,\OO_{C_n})\simeq\bigoplus_{j=0}^{n}H^1(C,\OO_C(j))$.
The latter inverse system is constant once $n\geq3$, by curve Serre
duality. Consequently the geometric genus of the cone is
\begin{align}\label{eq:pg}
 p_g&=\sum_{j\geq0}h^1(C,\OO_C(j))\notag\\
    &=\sum_{j=0}^3h^0(C,\OO_C(3-j))
      =10+6+3+1=20.
\end{align}
The completed local identifications of
Lemma~\ref{lem:singular-points} identify the corresponding formal
blow-ups and their exceptional thickenings: blowing up the maximal ideal
commutes with flat completion. The pullback of that ideal in the cone
model is $\mathcal J$. The theorem on formal functions
\cite[Chapter III, Theorem 11.1]{Hartshorne1977} therefore identifies the
stable cohomology just computed with the completed local contribution
of $R^1\rho_*\OO_Y$. That contribution has finite length, so its
completion is itself. The additivity above comes from the split cone
model, not from normality alone.
Normality is used separately to give $\rho_*\OO_Y=\OO_W$; the fibres
have dimension at most one, so there are no higher contributions. The Leray Euler characteristic
therefore gives
\begin{equation}\label{eq:chiY}
 \chi(\OO_Y)=4806-144\cdot20=1926.
\end{equation}
This agrees with the six-line covering calculation in
\cite[Remark 4.2 and Corollary 4.3]{Anghel2026}. In particular,
\[
 K_Y^2=17280,\qquad c_2(Y)=12\cdot1926-17280=5832,\qquad
 \sigma(Y)=17280-8\cdot1926=1872.
\]
These are also the numbers displayed for $n=6$ in the cited
Corollary~4.3.

\subsection{Global generation and ampleness of the polarization}

\begin{lemma}\label{lem:polarization}
The divisor $H=3A-E$ is globally generated and ample.
\end{lemma}

\begin{proof}
The line bundle $A$ is globally generated. Consider the three quadratic
coordinate products
\[
 u_0u_5,\qquad u_1u_4,\qquad u_2u_3.
\]
At every singular point of $W$, each product has exactly one inactive
factor. On the resolution that factor vanishes to first order along
the exceptional curve. Dividing the pullbacks by the canonical section
of $\OO_Y(E)$ thus gives three global sections of $2A-E$.

We show that the first two of these sections have no common zero.
Away from the exceptional divisor, a common zero would have
$u_0u_5=u_1u_4=0$. The four possibilities obtained by choosing one
factor from each product force, using the defining equations, an
active support among the four in \eqref{eq:singular-supports}.
Thus all such points of $W$ are singular.
On any exceptional sextic, the leading terms of the two products
are nonzero constants times two distinct inactive coordinates.
Two coordinates cannot vanish simultaneously on the diagonal plane
sextic: its remaining sixth-power coefficient is nonzero. Hence the
residual sections have no common zero on the exceptional curves either.
It follows that $2A-E$ is globally generated. Also,
\[
 (2A-E)^2=4A^2-4A\cdot E+E^2=864-0-864=0.
\]
Thus this nef divisor is not big. The first two residual sections span
a base-point-free pencil inside $|2A-E|$; this does not mean that the
complete linear system is a pencil. Indeed, the three displayed
quadratic products are linearly independent, since the homogeneous
ideal of the complete intersection $W$ has no degree-two part.
Dividing their pullbacks by the same exceptional factor preserves this
independence. If the residual sections are denoted $s_0,s_1,s_2$, their
relation is
\[
 s_0^6-s_1^6+s_2^6=0,
\]
not a linear relation: it follows from
$x(y-z)-y(x-z)+z(x-y)=0$ with
$x=u_0^6$, $y=u_1^6$, $z=u_2^6$.

Consequently $H=A+(2A-E)$ is globally generated, and both summands are
nef. For an irreducible curve $C$ not contracted by $\rho$, the ample
bundle $\OO_W(1)$ has positive degree on its image, so $A\cdot C>0$.
Since $2A-E$ is nef,
\[
 H\cdot C=A\cdot C+(2A-E)\cdot C\geq A\cdot C>0.
\]
The only contracted curves are the $E_i$, and
\[
 H\cdot E_i=-E_i^2=6>0.
\]
Finally, $H^2=1080>0$. The Nakai--Moishezon criterion for a projective
surface now gives ampleness; see \cite[Chapter V, Theorem 1.10]{Hartshorne1977}.
\end{proof}

\subsection{Application to the completed section ring}

\begin{proof}[Proof of Consequence~\ref{cor:example}]
Lemmas~\ref{lem:singular-points} and \ref{lem:polarization}, together
with \eqref{eq:H-square}, \eqref{eq:canonical-example}, and
\eqref{eq:chiY}, verify all hypotheses of Main claim~\ref{thm:main}.
The numerical margin is
\[
 c(Y,H)=\frac{15}{8}\cdot1080-1926=2025-1926=99.
\]
Main claim~\ref{thm:main} gives \eqref{eq:99} and the asserted
nonexistence. Taking $M=R$ gives $H^2_\mm(R)\ne0$.
Normality implies $\depth R\geq2$, so in fact $\depth R=2$.
The punctured spectrum is regular by Proposition~\ref{prop:cone},
and hence the singularity is isolated.

We also check the quasi-Gorenstein assertion. The graded canonical
module of a full section ring of a smooth projective surface is
\[
 \omega_B=\bigoplus_{t\in\Z}H^0(Y,\omega_Y\otimes L^t).
\]
This description does not require $B$ to be Cohen--Macaulay. It
follows from the degreewise identification of the top local cohomology
with $H^2(Y,L^t)$ and Serre duality; compare
\cite[Corollary 4.3]{Anghel2026}. Since $\omega_Y\simeq L^4$ and
negative powers of $L$ have no sections, it gives
$\omega_B\simeq B(4)$. After localization and completion the grading
shift is forgotten, so $\omega_R\simeq R$. Thus $R$ is
quasi-Gorenstein. It is not Gorenstein, because its depth is $2$.
\end{proof}

\subsection{Comparison with cyclic-deficiency criteria}
Let $E_R(\C)$ be an injective hull of the residue field, and put
\[
 K^2(R)=\Hom_R(H^2_\mm(R),E_R(\C)).
\]
This is the second deficiency module, which is finite by local duality.
The next observation does not use Main claim~\ref{thm:main} or the
asserted nonexistence conclusion.

\begin{lemma}\label{lem:noncyclic-deficiency}
For the explicit surface and ring in Consequence~\ref{cor:example},
\[
 \mu_R(K^2(R))\geq h^1(Y,\OO_Y)\geq10,
\]
where $\mu_R$ denotes the minimal number of generators. In particular,
$K^2(R)$ is not cyclic; moreover, $\len_RH^2_\mm(R)\geq234$.
\end{lemma}

\begin{proof}
Let $C\subset\PP^2$ be the smooth plane sextic
$v_2^6-v_0^6+v_1^6=0$, of genus $10$.
The rational map
\[
 W\dashrightarrow C,\qquad
 [u_0:\cdots:u_5]\longmapsto[u_0:u_1:u_3]
\]
is undefined precisely at the singular points with active support
$\{2,4,5\}$. At each such point, $u_0,u_1,u_3$ are local generators
of the maximal ideal, as in the local calculation in
Lemma~\ref{lem:singular-points}. Their pullbacks to the blow-up have
common exceptional factor of order one, and their residual sections
have no common zero on the exceptional sextic. Thus the map extends
to a morphism $f:Y\to C$. On each of these exceptional curves it
is the identity in the displayed plane coordinates, so $f$ is
surjective.

Factor $f$ as $Y\xrightarrow{g}C'\xrightarrow{h}C$ by Stein
factorization. Here $C'$ is a smooth projective curve,
$g_*\OO_Y=\OO_{C'}$, and $h$ is finite.
The Leray sequence injects $H^1(C',\OO_{C'})$ into $H^1(Y,\OO_Y)$.
In characteristic zero, $\operatorname{tr}_h/\deg(h)$ splits the
inclusion $\OO_C\to h_*\OO_{C'}$. Consequently
$h^1(Y,\OO_Y)\geq g(C)=10$.

Put $\mathcal H=H^2_\bb(B)$. The graded local-cohomology description
of the full section ring gives
$\mathcal H_t\simeq H^1(Y,L^t)$.
For $t>4$, the identity $L^t\simeq\omega_Y\otimes L^{t-4}$ and
Kodaira vanishing give $\mathcal H_t=0$; see
\cite[Chapter III, Remark 7.15, pp.\ 248--249]{Hartshorne1977}.
For $t<0$, first use Serre duality to obtain
\[
 H^1(Y,L^t)^*\simeq H^1(Y,\omega_Y\otimes L^{-t})
                  \simeq H^1(Y,L^{4-t})=0,
\]
where the last equality is the positive-twist case just proved, since
$4-t>4$. Thus $\mathcal H_t=0$ outside $0\leq t\leq4$.
All its graded pieces are finite-dimensional, so $\mathcal H$ has
finite length, and localization and completion identify it with
$H^2_\mm(R)$. For every $a>0$, multiplication by $B_a$ sends
$\mathcal H_4$ into $\mathcal H_{4+a}=0$. Its degree-four component is
therefore annihilated by all of $B_{>0}$ and has dimension
\[
 \dim_\C\mathcal H_4=h^1(Y,\omega_Y)=h^1(Y,\OO_Y).
\]
It is therefore contained in the socle of $H^2_\mm(R)$.
Duality for finite-length modules gives
$\mu_R(K^2(R))=\dim_\C\operatorname{Soc}_R H^2_\mm(R)$,
proving the generator bound.
Finally, surface Riemann--Roch gives
$\chi(L^2)=\chi(\OO_Y)-2H^2=-234$, and Serre duality with
$\omega_Y\simeq L^4$ gives $h^2(L^2)=h^0(L^2)$. The degree-two
summand therefore yields the independent estimate
\[
 \len_R H^2_\mm(R)\geq h^1(Y,L^2)
 =2h^0(Y,L^2)+234\geq234.
\]
\end{proof}

\begin{remark}\label{rem:existence-criteria}
Shimomoto--Tavanfar \cite[Theorem 3.2]{ShimomotoTavanfar2023}
prove an existence theorem for quasi-Gorenstein deformations of
three-dimensional quasi-Gorenstein Buchsbaum local rings with
$\len H^2=1$. For our three-dimensional ring the deformation sequence
would be empty, and this length-one hypothesis fails by
Lemma~\ref{lem:noncyclic-deficiency}.
Xie's criterion \cite[Theorem 1.1]{Xie2026} assumes a dualizing complex,
depth one below the dimension, and a cyclic deficiency module.
Although the first two conditions hold here, the cyclicity condition
fails by the same lemma. These are failures of specific sufficient
hypotheses, not alternative proofs of nonexistence.

The graded criterion of
\cite[Corollary 4.5]{ShimomotoTavanfar2023} is a separate,
positive-characteristic result over an $F$-finite field, formulated
there for projective varieties of dimension at least $3$.
It does not apply to the present complex surface. In particular,
its characteristic restriction must not be confused with the scope
of their characteristic-free local theorem just discussed.
\end{remark}

\section{Scope, provenance, and verification}\label{sec:scope}

The proposed conclusion concerns nonzero finite modules of arbitrary
rank over the completed local domain. It is stronger than graded
nonexistence, and it does not follow merely from the latter.
The known polarized surface is credited to the sources indicated in
Section~\ref{sec:example}; the additional local argument is the
combination of Propositions~\ref{prop:normalization},
\ref{prop:euler}, \ref{prop:tails}, and \ref{prop:persistence}.

Two possible shortcuts are deliberately avoided. First, the restriction
of a reflexive extension to $Y$ is only asserted to be torsion-free,
not locally free. Second, the direct-limit argument controls kernels
from a fixed cohomology module, rather than asserting that cohomology
commutes with a degeneration to a graded module. These distinctions
are needed for the intended non-graded conclusion.

The exact check reproduced in Appendix~\ref{app:checks} checks finitely
many matrix ranks, coordinate supports, and numerical identities. It does not quantify over modules,
verify the sheaf-theoretic lemmas, or provide a formal proof certificate.
This version is circulated for independent mathematical verification of
Main claim~\ref{thm:main} and its proposed proof. The auxiliary results
are stated separately, and Consequence~\ref{cor:example} is explicitly
conditional on the main claim. Compilation and successful execution of
the finite checks do not certify the main claim.

\appendix
\section{Exact checks for the explicit model}\label{app:checks}

The following finite procedure provides a reproducible check of the
coordinate-support classification in \eqref{eq:singular-supports}.
For each of the $63$ nonempty subsets $J$ of the columns of
\eqref{eq:matrix}, compute $(C_0)_J$ and its kernel over $\mathbb Q$.
A linear subspace over the infinite field $\C$ contains a vector with
every coordinate nonzero if and only if it is not contained in any
coordinate hyperplane. Thus a support is relevant precisely when every
coordinate functional is nonzero on the kernel. Retain those supports
for which the matrix has rank less than $3$. The result consists of the
four triples in \eqref{eq:singular-supports}. The one-dimensional left
kernel then gives the residual coefficients listed in
Section~\ref{sec:example}.

The same calculation checks the polynomial Euler characteristics in
\eqref{eq:chiW}, the plane-curve sum in \eqref{eq:pg}, and the tuple
\[
 (A^2,E^2,H^2,K_Y^2,c_2(Y),\chi(\OO_Y),c(Y,H))
 =(216,-864,1080,17280,5832,1926,99).
\]
There is also a covering-stratification check of $c_2(Y)$, using the
six-line construction of \cite[Section 2.2]{Hirzebruch1983}. For a
covering of degree $N=6^5=7776$, with ten rational branch components
and fifteen double points on the blown-up plane, the strata have
Euler characteristics $2$, $-10$, and $15$. The respective numbers
of sheets are $N$, $N/6$, and $N/36$. Hence
\[
 c_2(Y)=2N-10\frac N6+15\frac N{36}=5832,
\]
consistent with $12\chi(\OO_Y)-K_Y^2$.

\subsection*{A self-contained exact-check listing}
The listing below combines the finite checks in one place. It is ordinary
Python~3 code using SymPy for exact rational and symbolic algebra; it
reads and writes no files. Copying the listing into a Python session
with SymPy installed is sufficient to run the check. The listing is
printed for reproducibility and is not executed when this manuscript
is compiled. None of its assertions verifies the general
sheaf-theoretic argument or quantifies over modules.

\begin{lstlisting}
from itertools import combinations
import sympy as sp

C = sp.Matrix([
    [-1,  1, 0, 1, 0, 0],
    [-1,  0, 1, 0, 1, 0],
    [ 0, -1, 1, 0, 0, 1],
])
expected = [(0, 1, 2), (0, 3, 4),
            (1, 3, 5), (2, 4, 5)]
records, examined = [], 0
for size in range(1, 7):
    for active in combinations(range(6), size):
        examined += 1
        sub = C[:, active]
        kernel = sub.nullspace()
        meets_torus = bool(kernel) and all(
            any(v[j] != 0 for v in kernel)
            for j in range(size)
        )
        rank = sub.rank()
        if not (meets_torus and rank < 3):
            continue
        inactive = tuple(i for i in range(6)
                         if i not in active)
        left = sub.T.nullspace()
        assert size == 3 and rank == 2
        assert len(kernel) == len(left) == 1
        relation = left[0].T * C
        assert all(relation[0, i] == 0 for i in active)
        coeff = tuple(relation[0, i] for i in inactive)
        assert all(c != 0 for c in coeff)
        records.append((active, inactive, coeff))
assert examined == 63
assert [rec[0] for rec in records] == expected
points = len(records) * 6**2
assert points == 144

def chi_p5(t):
    # Polynomial binomial coefficient, also for negative t.
    return sp.prod(t + i for i in range(1, 6)) / sp.Integer(120)

chi_W = sum((-1)**j * sp.binomial(3, j) * chi_p5(-6*j)
            for j in range(4))
pg = sum(sp.binomial(5-j, 2) for j in range(4))
chi_Y = chi_W - points * pg
A2, E2 = 6**3, -6 * points
H2 = 9*A2 + E2
K2 = 16*H2
c2 = 12*chi_Y - K2
margin = sp.Rational(15, 8)*H2 - chi_Y
sigma = K2 - 8*chi_Y
assert margin == (sigma-H2)/8
assert 4*A2 + E2 == 0
assert sigma == 1872
assert (chi_W, pg, chi_Y) == (4806, 20, 1926)
assert (A2, E2, H2, K2, c2, chi_Y, margin) == (
    216, -864, 1080, 17280, 5832, 1926, 99
)

# Independent arithmetic route, conditional on the branch data.
N = 6**5
c2_cover = 2*N - 10*(N//6) + 15*(N//36)
H2_cover = sp.Rational(5*N, 36)
chi_cover = (16*H2_cover + c2_cover) / 12
assert (N, c2_cover, H2_cover, chi_cover) == (
    7776, 5832, 1080, 1926
)
assert c2_cover == c2 and chi_cover == chi_Y

# Check the sixth-power relation for the opposite-edge sections.
p, q, t = sp.symbols("p q t")
assert sp.expand(p*(q-t) - q*(p-t) + t*(p-q)) == 0

# Check the two elementary polynomial identities used in the text.
x = sp.symbols("x")
assert sp.expand(1926 + 540*(x*x-4*x)
                 - (-234 + 540*(x-2)**2)) == 0
a0, a1, a2 = sp.symbols("a0 a1 a2")
assert sp.expand((a0-a1-a2)
                 - ((a0-2*a1+a2)+(a1-2*a2))) == 0
for active, inactive, coeff in records:
    print("support:", active, "residual:", inactive, coeff)
print("supports checked:", examined, "singular points:", points)
print("(A2,E2,H2,K2,c2,chi,margin):",
      tuple(map(int, (A2, E2, H2, K2, c2, chi_Y, margin))))
print("Finite checks passed; no general proof certificate.")
\end{lstlisting}

All matrix ranks, kernel calculations, and polynomial identities in
this listing are exact. Its numerical assertions use the geometric
interpretations proved in Section~\ref{sec:example}; they do not
independently establish those interpretations.

\section*{Declaration of generative AI assistance}
ChatGPT (OpenAI) was used as a research and writing aid in developing and
critically examining the mathematical arguments, preparing the computational
checks in Appendix~\ref{app:checks}, identifying relevant literature, and
organizing, drafting, and editing the manuscript. This assistance extended
beyond language polishing. The author takes full responsibility for all
content, including the mathematical arguments, computations, code, and
references. No AI system is an author, and AI-generated outputs do not
constitute independent verification of the mathematical results.

\section*{Data and code availability}
No empirical dataset is used. All finite input data, expected numerical
results, and the complete exact-check listing are contained in
Appendix~\ref{app:checks}. No supplementary file is required to read or
compile this manuscript. The geometric and sheaf-theoretic arguments
are given in the body of the paper and are not replaced by these checks.

\enlargethispage{2\baselineskip}

\end{document}